\documentclass[12pt,a4paper]{article}
\usepackage{amsmath,amsthm,amsfonts} 
\usepackage{txfonts} 

\newcommand{\projective}{\newlength{\projectiveheight}%
\settoheight{\projectiveheight}{$\scriptstyle\wedge$}%
$(\overline{\raisebox{0mm}[0.55\projectiveheight][0mm]{$\scriptstyle\wedge$}})$}

\DeclareMathOperator\dis{\triangle}%
\DeclareMathOperator\E{E}%
\DeclareMathOperator\GL{GL}%
\DeclareMathOperator\GE{GE}%
\DeclareMathOperator\Harm{H}%
\DeclareMathOperator\Seq{\mathcal S}
\newcommand{\Trans}{\mathrm T}

\newcommand{\bP}{{\mathbb P}}

\newcommand{\XX}{\phantom{-}}

\newtheorem{thm}{Theorem}
\newtheorem{lem}{Lemma}
{\theoremstyle{definition}
\newtheorem{exa}{Example}
}

\begin{document}

\title{Von Staudt's theorem revisited}

\author{Hans Havlicek}

\maketitle

\begin{abstract}
We establish a version of von Staudt's theorem on mappings which preserve
harmonic quadruples for projective lines over (not necessarily commutative)
rings with ``sufficiently many'' units, in particular $2$ has to be a unit.

\end{abstract}

\section{Introduction}\label{se:intro}

The first edition of the seminal book \emph{Geometrie der Lage} by Karl Georg
Christian von Staudt appeared in 1847; see \cite{hart-08a} for publication
details. Projectivities are defined there by the invariance of harmonic
quadruples \cite[p.~49]{stau-47}: \emph{''Zwei einf\"{o}rmige Grundgebilde heissen
zu einander projektivisch {\projective}, wenn sie so auf einander bezogen sind,
dass jedem harmonischen Gebilde in dem einen ein harmonisches Gebilde im andern
entspricht.''} Next, after defining perspectivities, the following theorem is
established: Any projectivity is a finite composition of perspectivities and
vice versa. (It was noticed later that there is a small gap in von Staudt's
reasoning. A detailed exposition can be found in \cite{voel-08a}.) Any result
in this spirit now is called a \emph{von Staudt's theorem}.
\par
In the present article we shall be concerned with \emph{projective lines over
rings} (associative with a unit element) and the algebraic description of their
\emph{harmonicity preservers}, i.~e., mappings which take all harmonic
quadruples of a first projective line to harmonic quadruples of a second one.
There is a widespread literature on this topic. The following short review is
rather sketchy, as it does not fully reflect the varying (often rather
technical) assumptions on the underlying rings. Part of the presented material
is related with mappings which reappear in a more general setting in the
surveys \cite{brehm-08a} and \cite{veld-95a}.
\par
All harmonicity preserving bijections of the projective line over any
\emph{commutative field} $F$ of characteristic $\neq 2$ onto itself were
determined by O.~Schreier and E.~Sperner \cite[p.~191]{schrei+s-35a}. In terms
of an underlying $F$-vector space $V$ these transformations comprise precisely
the \emph{projective semilinear group} $\mathrm{P}\Gamma\mathrm{L}(V)$. The
case of a (not necessarily commutative) \emph{field} of characteristic $\neq 2$
was settled in several steps by G.~Ancochea \cite{anco-41a}, \cite{anco-42a},
\cite{anco-47a} and L.-K.~Hua \cite{hua-49a} (see also \cite{hua-52a}). For a
proper skew field $F$ one has to include mappings which arise from
\emph{antiautomorphisms} of $F$ (provided that $F$ admits any
antiautomorphism). A.~J.~Hoffman \cite{hoff-51a} ($F$ commutative) and R.~Baer
\cite[p.~78]{baer-52a} ($F$ arbitrary) proved that similar results hold if the
invariance of harmonic quadruples is replaced by the invariance of an arbitrary
cross ratio $k\neq 0,1$ in the centre of $F$. In this way the case of
characteristic $2$ need no longer be excluded. A detailed account with
historical remarks is given in \cite[pp.~56--57]{karz+k-88a}.
\par
There are several outcomes for the projective line over a \emph{ring $R$ with
stable rank $2$}: Loosely speaking, in the case of a \emph{commutative ring}
$R$ the result of Schreier and Sperner remains unaltered provided that $R$
contains ``sufficiently many'' units, in particular $2$ has to be a unit in
$R$. Contributions (under varying additional assumptions) are due to W.~Benz
\cite{benz-64a}, \cite[pp.~173--183]{benz-73a}, B.~V.~Limaye and N.~B.~Limaye
\cite{lima+l-77b}, N.~B. Limaye \cite{lima-71a}, \cite{lima-72a},
B.~R.~McDonald \cite{mcdon-81a}, and H.~Schaef\-fer \cite{schae-74a}. Little
seems to be known for non-commutative rings: B.~V.~Limaye and N.~B.~Limaye
(\cite{lima+l-77a}, \cite{lima+l-77c}) treated the case of a (not necessarily
commutative) \emph{local ring} $R$. They determined all bijections of the
projective line over $R$ such that all quadruples with a given cross ratio $k$
go over to quadruples with a given cross ratio $k'$, where $k, k'$ are elements
in the centre of $R$ other that $0,1$. Here the algebraic description is more
involved, since one has to use \emph{Jordan automorphisms} (or, in a different
terminology, \emph{semiautomorphisms}) of $R$. More information can be
retrieved from the surveys in \cite{bart+b-85}, \cite{benz+l+s-72a}, and
\cite{benz+s+s-81}.
\par
F.~Buekenhout \cite{buek-65a}, St.~P.~Cojan \cite{cojan-85a}, D.~G.~James
\cite{james-82a}, and B.~Klotzek \cite{klotz-88a} characterised those (not
necessarily injective) mappings between projective lines over fields which
satisfy a much weaker form of cross ratio preservation than the one mentioned
in the preceding paragraph. The link with ring geometry is achieved via a
recoordinatisation of the domain projective line in terms of a \emph{valuation
ring} \cite{james-82a}.
\par
It was pointed out by C.~Bartolone and F.~Di~Franco \cite{bart+f-79} that an
algebraic description of all harmonicity preserving bijections of the
projective line over an arbitrary ring is out of reach, even in the commutative
case. They therefore initiated the study of mappings which preserve
\emph{generalised harmonic quadruples} and succeeded in describing all such
mappings for commutative rings; see also M.~Kulkarni \cite{kulk-80a}. However,
this goes beyond the scope of the present article. With regard to the
non-commutative case, we refer to the work of C.~Bartolone and F.~Bartolozzi
\cite{bart+b-85}, D.~Chkhatarashvili \cite{chkh-98a}, L.~Cirlincione and
M.~Enea \cite{cirl+e-90}, and A.~A.~Lashkhi \cite{lash-89a}, \cite{lash-90a},
\cite{lash-97a}. Take notice that some of the quoted papers are merely short
communications without any proof. For harmonicity preserving mappings of other
geometric structures see \cite{bert-03a}, \cite{blunck-91a}, \cite{ferr-81a},
and the references therein. It is also worth noting that the invariance of
harmonic quadruples appears together with other conditions in an early paper
\cite{hua-45a} of L.-K.~Hua on a characterisation of certain transformations of
matrix spaces. However, as Hua pointed out in a subsequent note \cite{hua-45b},
the condition about harmonic quadruples is superfluous in that context, and it
afterwards disappeared from the so-called \emph{geometry of matrices}; cf.\ the
monographs \cite{huanglp-06a} and \cite{wan-96a}. An analogous result for
projective lines over certain semisimple rings is due to A.~Blunck and the
author \cite{blunck+h-05b}.
\par
The present article is organised as follows: In Section \ref{se:basic} we
collect the relevant notions and we recall the definition of harmonicity
preservers which arise from Jordan homomorphisms. Our main result is
Theorem~\ref{thm:staudt} in Section~\ref{se:staudt}. It shows that under
certain conditions there are no other harmonicity preservers between projective
lines over rings, but those which arise from Jordan homomorphisms. A major tool
in our proof is a lemma from \cite{lima+l-77b} which characterises Jordan
homomorphisms.

\section{Basic notions and examples}\label{se:basic}

All our rings are associative with a unit element $1$ which is inherited by
subrings and acts unitally on modules. The trivial case $1=0$ is excluded. The
group of units (invertible elements) of a ring $R$, say, will be denoted by
$R^*$.
\par
Let $R$ be a ring and let $M$ be a free left $R$-module of rank $2$. We say
that $a\in M$ is \emph{admissible} if there exists $b\in M$ such that $(a,b)$
is a basis of $M$ (with two elements). As a matter of fact, we do not require
that all bases of $M$ have the same number of elements; cf.\
\cite[p.~3]{lam-99a}.
\par
The following exposition is mainly taken from \cite[p.~785]{herz-95a}; see also
\cite[pp.~15--16]{blunck+he-05a} or \cite[pp.~899--904]{havl-12b}: The
\emph{projective line} over $M$ is the set $\bP(M)$ of all cyclic submodules
$Ra$, where $a\in M$ is admissible. The elements of $\bP(M)$ are called
\emph{points}. At times it will be convenient to use coordinates with respect
to some basis $(e_0,e_1)$ of $M$. Given any pair $(a,b)\in M^2$ let $(x_0,x_1)$
and $(y_0,y_1)$ be the coordinates of $a$ and $b$, respectively. The matrix
\begin{equation}\label{eq:matrix}
    \begin{pmatrix}x_0&x_1\\y_0&y_1\end{pmatrix}
\end{equation}
will be called the \emph{matrix of $(a,b)$ w.~r.~t.\ the basis $(e_0,e_1)$}.
The pair $(a,b)$ is a basis of $M$ if, and only if, the matrix in
(\ref{eq:matrix}) is invertible. Thus $(x_0,x_1)\in R^2$ is admissible (or,
said differently, a coordinate pair of a point) precisely when it is the first
(or second) row of a matrix in $\GL_2(R)$. One particular case deserves
explicit mention, since it links the group $R^*$ with the group $\GL_2(R)$: For
all $x,y\in R$ holds
\begin{equation}\label{eq:GL-R*}
    \begin{pmatrix}x&1\\y&1\end{pmatrix}\in\GL_2(R)
    \quad\mbox{if, and only if,}\quad
    x-y\in R^* .
\end{equation}
This is immediate from
\begin{equation}
    \begin{pmatrix}1&1\\0&1\end{pmatrix}
    \begin{pmatrix}x-y&0\\0&1\end{pmatrix}
    \begin{pmatrix}1&0\\y&1\end{pmatrix} =
    \begin{pmatrix}x&1\\y&1\end{pmatrix}.
\end{equation}
\par
By definition, each point $p\in\bP(M)$ has an admissible generator, say $a$. If
there exist $x,y\in R$ with $xy=1$ and $yx\neq 1$ then $ya$ is a non-admissible
generator of $p$, whereas $xa$ is an admissible generator of a point other than
$p$ \cite[Prop.~2.1 and Prop.~2.2]{blunck+h-00b}. We adopt from now on the
following convention: \emph{We only use admissible generators of points.} Two
admissible elements of $M$ generate the same point precisely when they are
left-proportional by a unit in $R$.
\par
Two points $p$ and $q$ are called \emph{distant}, in symbols $p \dis q$, if
$M=p\oplus q$. For all $a,b\in M$ holds $Ra \dis Rb$ precisely when the
coordinate matrix of $(a,b)$ w.~r.~t.\ any basis $(e_0,e_1)$ of $M$ is
invertible. The graph of the relation $\dis$, i.~e. the pair
$\big(\bP(M),{\dis}\big)$, is called the \emph{distant graph\/} of $\bP(M)$. It
is an undirected graph without loops, and it need not be connected. In order to
describe the \emph{connected components} of the distant graph we need some
prerequisites.
\par
The \emph{elementary linear group\/} $\E_2(R)$ is generated by the set of all
matrices
\begin{equation*}\label{eq:E(t)}
  E(t):=\begin{pmatrix}
  \XX t&1\\-1&0
  \end{pmatrix}
  \quad\mbox{with}\quad  t\in R ;
\end{equation*}
see \cite[p.~5]{cohn-66a}. Let $\Seq(R)$ be the set of all finite sequences in
$R$ (including the empty sequence). We adopt the shorthand notation
\begin{equation*}\label{eq:E(T)}
    E(T):= E(t_1)\cdot E(t_2)\cdots E(t_n)
    \quad\mbox{where}\quad T=(t_1,t_2,\ldots,t_n)\in\Seq(R).
\end{equation*}
(Note that $n\geq 0$, the length of $T$, is arbitrary.) From
$E(t)^{-1}=E(0,-t,0)$ follows that all matrices $E(T)$ with $T\in\Seq(R)$
comprise the entire group $\E_2(R)$. The subgroup of $\GL_2(R)$, which is
generated by $\E_2(R)$ and the set of all invertible diagonal matrices, is
denoted by $\GE_2(R)$. By definition, a \emph{$\GE_2$-ring\/} $R$ is
characterised by $\GL_2(R)=\GE_2(R)$.
\par
If $(e_0,e_1)$ is a basis of $M$ then the connected component of the point
$Re_0\in\bP(M)$ is given by the set of all points $p=R(x_0e_0+x_1e_1)$, where
$(x_0,x_1)$ is the first row of some matrix $E(T)$ with $T\in\Seq(R)$ or, said
differently,
\begin{equation}\label{eq:component}
    (x_0,x_1)=(1,0)\cdot E(T) \quad\mbox{for some}\quad  T\in\Seq(R).
\end{equation}
Furthermore, the distant graph $\big(\bP(M),{\dis}\big)$ is connected precisely
when $R$ is a $\GE_2$-ring \cite[Thm.~3.2]{blunck+h-01a}.
\par
A quadruple $(p_0,p_1,p_2,p_3)\in\bP(M)^4$ is \emph{harmonic} if its cross
ratio \cite[p.~787]{herz-95a} equals $-1\in R$, i.~e., there exists a basis
$(g_0,g_1)$ of $M$ such that
\begin{equation}\label{eq:H.g}
    p_0=R g_0,\quad p_1=R g_1,\quad  p_2=R(g_0+g_1),\quad p_3=R(g_0-g_1).
\end{equation}
In this case we write $\Harm(p_0,p_1,p_2,p_3)$. In terms of coordinates
w.~r.~t.\ some basis $(e_0,e_1)$ of $M$ there is an alternative description:
$\Harm(p_0,p_1,p_2,p_3)$ holds if, and only if, there is a matrix
$G\in\GL_2(R)$ such that
\begin{equation}\label{eq:H.G}
    (1,0)\cdot G,\quad
    (1,0)\cdot E(0)\cdot G,\quad
    (1,0)\cdot E(1)\cdot G,\quad
    (1,0)\cdot E(-1)\cdot G
\end{equation}
are coordinates of the points $p_0$, $p_1$, $p_2$, $p_3$, respectively. Indeed,
if (\ref{eq:H.g}) holds for some basis $(g_0,g_1)$ we can take as $G$ the
coordinate matrix of $(g_0,g_1)$ w.~r.~t.\ $(e_0,e_1)$ in order to obtain
(\ref{eq:H.G}). Conversely, the rows of $G$ provide the coordinates w.~r.~t.\
$(e_0,e_1)$ of an appropriate basis of $M$ to guarantee
$\Harm(p_0,p_1,p_2,p_3)$.
\par
From $\Harm(p_0,p_1,p_2,p_3)$ follows $p_0\dis p_1$ and $p_i\dis p_j$ for all
$i\in\{0,1\}$ and all $j\in\{2,3\}$. \emph{Therefore all four points belong to
the same connected component of the distant graph $(\bP(M),\dis)$.} By virtue
of (\ref{eq:GL-R*}), we have
\begin{equation}\label{eq:2}
    p_2\dis p_3
    \quad\mbox{if, and only if,}\quad 2\in R^* .
\end{equation}
The inequality $p_2\neq p_3$ holds precisely when $-1\neq 1\in R$. (In
\cite[4.7]{blunck+h-03b} these two conditions erroneously got mixed up.)
\par
If $p_0,p_1,p_2$ are three mutually distant points of $\bP(M)$ then there is a
unique point of $\bP(M)$, say $p_3$ with $\Harm(p_0,p_1,p_2,p_3)$. This is the
well known \emph{uniqueness of the fourth harmonic point}. Since
$\Harm(p_0,p_1,p_2,p_3)$ is equivalent to $\Harm(p_0,p_1,p_3,p_2)$, there holds
as well the \emph{uniqueness of the third harmonic point}. The latter (less
prominent) property will be used when proving Lemma~\ref{lem:2}.
\par
Let $M'$ be a free left module of rank $2$ over a ring $R'$. A mapping
$\mu:\bP(M)\to\bP(M')$ will be called a \emph{harmonicity preserver} if it
takes all harmonic quadruples of $\bP(M)$ to harmonic quadruples of $\bP(M')$.
No further assumptions, like injectivity or surjectivity of $\mu$ are made
here. A simple, though important, property is that any harmonicity preserver
$\mu:\bP(M)\to\bP(M')$ is \emph{distant preserving}, i.~e.,
\begin{equation}\label{eq:dist-pr}
    p_0\dis p_1
    \quad\mbox{implies}\quad
    p_0^\mu \dis p_1^\mu
    \quad\mbox{for all}\quad
    p_0,p_1\in\bP(M).
\end{equation}
This follows readily from the existence of points $p_2$ and $p_3$ with
$\Harm(p_0,p_1,p_2,p_3)$.
\par
We close this section by quoting several examples of harmonicity preservers
$\bP(M)\to\bP(M')$.

\begin{exa}\label{exa:jordan}
Let $\alpha:R\to R'$ be a Jordan homomorphism, i.~e. a mapping satisfying
\begin{equation*}\label{eq:jordan}
  (x+y)^\alpha = x^\alpha + y^\alpha,\quad
  1^\alpha = 1',\quad
  (xyx)^\alpha = x^\alpha y^\alpha x^\alpha\quad
  \mbox{for all}\quad x,y\in R.
\end{equation*}
See, among others, \cite[p.~832]{herz-95a} or \cite[p.~2]{jacob-68a}. Also, let
$C$ be any connected component of the distant graph $(\bP(M),\dis)$. We select
bases $(e_0,e_1)$ and $(e_0',e_1')$ of $M$ and $M'$, respectively, subject to
the condition $Re_0\in C$. According to a result of A.~Blunck and the author
\cite[Thm.~4.4]{blunck+h-03b} the following (rather cumbersome) construction
gives a well defined mapping
\begin{equation}\label{eq:mu}
    \mu : C \to \bP(M') : p \mapsto p^\mu.
\end{equation}
By (\ref{eq:component}), any point $p\in C$ can be written in the form
$p=R(x_0e_0+x_1e_1)$ with
\begin{equation*}\label{eq:(x0,x1)}
    (x_0,x_1)=(1,0)\cdot  E(T)
\end{equation*}
for some $T\in\Seq(R)$, say $T=(t_1,t_2,\ldots,t_n)$ with $n\geq 0$. We use the
shorthand $T^\alpha:=(t_1^\alpha,t_2^\alpha,\ldots,t_n^\alpha)\in\Seq(R')$ and
let
\begin{equation}\label{eq:(x0',x1')}
    (x_0',x_1') := (1,0)\cdot E(T^\alpha).
\end{equation}
The point $p^\mu$ is defined as $R'(x_0'e_0'+x_1'e_1')$. By
\cite[Prop.~4.8]{blunck+h-03b}, $\Harm(p_0,p_1,p_2,p_3)$ implies
$\Harm(p_0^\mu,p_1^\mu,p_2^\mu,p_3^\mu)$ for all $p_0,p_1,p_2,p_3\in C$.
\par
The previous construction can be repeated for all connected components of the
distant graph on $\bP(M)$. Thereby is not necessary to stick to a fixed Jordan
homomorphism. Altogether this gives a globally defined harmonicity preserver
$\bP(M)\to \bP(M')$.
\par
One particular case, due to C.~Bartolone \cite{bart-89a}, deserves special
mention: Let $R$ be a ring of \emph{stable rank\/ $2$}
\cite[p.~1039]{veld-95a}. Then $\big(\bP(M),{\dis}\big)$ has a single connected
component, each of its points can be described in terms of at least one finite
sequence $T=(t_1,t_2)\in R^2$, and $\mu$ can be rewritten as
\begin{equation*}\label{eq:bart}
    \mu: \bP(M)\to \bP(M'): R\big((t_1t_2-1)e_0 + t_1 e_1\big)\mapsto
    R'\big((t_1^\alpha t_2^\alpha-1)e_0' + t_1^\alpha e_1'\big) .
\end{equation*}
\end{exa}

\begin{exa}\label{exa:hom}
We adopt the settings of Example~\ref{exa:jordan}, but we make the extra
assumption that $\alpha$ is a homomorphism of rings. Then
\begin{equation*}\label{eq:semilin}
    \sigma: M\to M' : x_0e_0 +x_1e_1 \mapsto x_0^\alpha e_0' +x_1^\alpha e_1'
    \quad\mbox{for all}\quad x_0,x_1\in R
\end{equation*}
is an $\alpha$-semilinear mapping and
\begin{equation*}\label{eq:alpha*}
    \alpha_* : \GL_2(R)\to\GL_2(R') : X \mapsto X^\alpha,
\end{equation*}
i.~e., $\alpha$ is applied to each entry of $X$, is a homomorphism of groups.
Thus for any basis $(a,b)$ of $M$ the image $(a^\sigma,b^\sigma)$ is a basis of
$M'$. Consequently, the mapping
\begin{equation*}\label{eq:lambda}
     \lambda:\bP(M)\to \bP(M') : Ra \mapsto R'(a^\sigma)
     \quad\mbox{(with~} a\in M
     \mbox{~admissible)}
\end{equation*}
is well defined, and it preserves harmonicity. The mapping $\mu$ from
(\ref{eq:mu}) is the \emph{restriction\/} of $\lambda$ to $C$. The matrix
$E(T^\alpha)$ from (\ref{eq:(x0',x1')}) now can be expressed as
$E(T)^{\alpha_*}$, since $E(t)^{\alpha_*}=E(t^\alpha)$ for all $t\in R$.
\end{exa}

\begin{exa}\label{exa:anti}
We adopt the settings of Example~\ref{exa:jordan}, but we make the extra
assumption that $\alpha$ is an antihomomorphism of rings. We have the
homomorphism
\begin{equation}\label{eq:alpha**}
    \alpha_{**} : \GL_2(R)\to\GL_2(R') :
    X \mapsto E(0)^{-1}\cdot\big((X^{-1})^{\Trans}\big)^\alpha\cdot E(0),
\end{equation}
where $(X^{-1})^\Trans$ denotes the transpose of $X^{-1}$ and $\alpha$ is
applied entrywise. (We must not use $\alpha_*$ in (\ref{eq:alpha**}), since
$(X^{-1})^\Trans$ need not be invertible.) A straightforward calculation shows
$E(t)^{\alpha_{**}}= E(t^\alpha)$ for all $t\in R$. Hence the matrix
$E(T^\alpha)$ from (\ref{eq:(x0',x1')}) now can be expressed as
$E(T)^{\alpha_{**}}$. This leads us to the definition of a mapping
\begin{equation*}\label{eq:delta}
    \delta : \bP(M)\to\bP(M') : R(x_0e_0+x_1e_1) \mapsto R'(x_0'e_0'+x_1'e_1')
\end{equation*}
which runs as follows: $(x_0,x_1)$ is chosen as the first row of any matrix
$X\in\GL_2(R)$ and $(x_0',x_1')$ is defined as the first row of the matrix
$X^{\alpha_{**}}$. By \cite[Ex.~4.8]{blunck+h-03b} this mapping is well
defined. An equivalent (and more lucid) definition of $\delta$ in terms of the
dual module of $M$ can be read off from \cite[Rem.~5.4]{blunck+h-01b} or
\cite[Prop.~3.3]{herz-87a}. Formula (\ref{eq:H.G}) provides an easy direct
proof for $\delta$ being a harmonicity preserver. The mapping $\mu$ from
(\ref{eq:mu}) is the \emph{restriction\/} of $\delta$ to $C$.
\par
It may happen that $\alpha:R\to R'$ is a homomorphism and an antihomomorphism.
Then $R^\alpha$ is a commutative subring of $R'$ and we have $(\det
X^{\alpha_{*}}) X^{\alpha_{**}} = X^{\alpha_{*}}$ for all $X\in\GL_2(R)$. So in
this case the mappings $\lambda$ and $\delta$ coincide.
\end{exa}

\section{Von Staudt's theorem}\label{se:staudt}

We already noted in Section~\ref{se:basic} that the distant graph on $\bP(M)$
has a single connected component if, and only if, $R$ is a $\GL_2$-ring. In
this case the following version of von Staudt's theorem provides a unified
algebraic description of harmonicity preservers, otherwise it gives only a
description on an arbitrarily chosen connected component.
\begin{thm}\label{thm:staudt}
\renewcommand{\theenumi}{\roman{enumi}}%
\renewcommand{\labelenumi}{{\rm (\theenumi)}}%
Let $M$ and $M'$ be free modules of rank $2$ over rings $R$ and $R'$,
respectively. Furthermore, let $R$ satisfy the two conditions:
\begin{enumerate}
\item\label{cond:i} Given $x_1,x_2,\ldots,x_5\in R$ there exists $x\in R$
    such that $x-x_1,x-x_2,\ldots,x-x_5$ are units in $R$.
\item\label{cond:ii} $2$ is a unit in $R$.
\end{enumerate}
Let $\mu:\bP(M)\to \bP(M')$ be a harmonicity preserver. Choose any connected
component, say $C$, of the distant graph $\big(\bP(M),{\dis}\big)$. Then there
exist a basis $(a_{0},a_{1})$ of $M$, a basis $(a_{0}',a_{1}')$ of $M'$, and a
Jordan homomorphism $\alpha:R\to R'$ such that the restriction of $\mu$ to $C$
admits the following description:
\begin{equation*}\label{eq:mu|C}
    \mu | C : C\to\bP(M') :
    R(x_0 a_{0} + x_1 a_{1})\mapsto R(x_0' a_{0}' + x_1' a_{1}') ,
\end{equation*}
where
\begin{equation}\label{eq:(x0,x1)(x0',x1')}
    (x_0,  x_1) = (1,0)\cdot E(T),
    \quad
    (x_0',  x_1') = (1,0)\cdot E(T^\alpha),
\end{equation}
and $T$ is any finite sequence of elements in $R$.
\end{thm}

We postpone the proof until we have established four auxiliary results. In all
of them we tacitly adopt the assumptions of Theorem~\ref{thm:staudt}.
Lemma~\ref{lem:1} is self-explanatory. In Lemma~\ref{lem:2} we exhibit a
mapping $\beta:R\to R'$ which can be viewed as a ``local coordinate
representation of $\mu$''. Next, in Lemma~\ref{lem:3}, we establish that ``new
local coordinates'' (describing other parts of the given projective lines) can
be chosen in such a way that the ``new local coordinate representations'' of
$\mu$ coincides with the ``old'' one. This observation is the backbone of our
demonstration. Afterwards, in Lemma~\ref{lem:4}, the mapping $\beta$ is shown
to be a Jordan homomorphism. The actual proof Theorem~\ref{thm:staudt} amounts
then to verifying that the given mapping $\mu|C$ coincides with the harmonicity
preserver which arises from $\beta$ according to Example~\ref{exa:jordan}. It
goes without saying that part of our demonstration follows the same lines as
previous work by other authors. Condition~(\ref{cond:i}) is taken from
\cite{lima+l-77b}. It is equivalent to the following property of the projective
line $\bP(M)$:
\begin{enumerate}
\item[(i')]\label{cond:i'} Given points $p_1,p_2,\ldots,p_5\in \bP(M)$, all
    of which are distant to some point $p_0\in\bP(M)$, there exists $p\in
    \bP(M)$ which is distant to $p_0,p_1,\ldots,p_5$.
\end{enumerate}
The equivalence follows easily from (\ref{eq:GL-R*}) upon choosing a basis
$(e_0,e_1)$ of $M$ with $p_0=Re_0$. Then $p_i=R(x_i e_0+e_1)$ for
$i\in\{1,2,\ldots,5\}$ and $p=R(x e_0+e_1)$. Take notice that neither the
elements $x_1,x_2,\ldots,x_5$ nor the points $p_1,p_2,\ldots,p_5$ are assumed
to be distinct.

\begin{lem}\label{lem:1}
$2$ is a unit in $R'$.
\end{lem}

\begin{proof}
Since $M$ is free of rank $2$, there exists a harmonic quadruple
$(p_0,p_1,p_2,p_3)$ in $\bP(M)^4$. We read off $p_2\dis p_3$ from (\ref{eq:2})
and (\ref{cond:ii}). Application of $\mu$ yields $p_2^\mu\dis p_3^\mu$ by
virtue of (\ref{eq:dist-pr}). Now (\ref{eq:2}) in turn shows that $2$ is a unit
in $R'$.
\end{proof}

\begin{lem}\label{lem:2}
Given bases $(e_0,e_1)$ of $M$ and $(e_0',e_1')$ of $M'$ such that
\begin{equation}\label{eq:H.e}
    (Re_0)^\mu = R'e_0',\quad (Re_1)^\mu = R'e_1',\quad
    \big(R(e_0\pm e_1)\big)^\mu = R'(e_0'\pm e_1')
\end{equation}
there exists a unique mapping $\beta:R\to R'$ with the property
\begin{equation}\label{eq:beta.e}
    \big(R(x e_0+ e_1)\big)^\mu = R'(x^\beta e_0'+ e_1')
    \quad\mbox{for all}\quad x\in R.
\end{equation}
This $\beta$ is additive and satisfies $1^\beta=1$.
\end{lem}

\begin{proof}
For any $x\in R$ the point $p:=R(x e_0+ e_1)$ is distant from $Re_0$. From
(\ref{eq:dist-pr}) follows $p^\mu \dis R'e_0'$ so that the point $p^\mu$ has a
unique generator of the form $x'e_0'+e_1'$ with $x'\in R'$. We therefore can
define a unique mapping $\beta:R \to R'$ satisfying condition (\ref{eq:beta.e})
by $x^\beta:=x'$.
\par
By (\ref{eq:GL-R*}), for all $x,y\in R$ with $x-y\in R^*$ the points
\begin{equation*}\label{eq:mitte}
\begin{aligned}
    &q_0:=R (x e_0+e_1),&
    &q_1:=R(y e_0+e_1),\\
    &q_2:=R\big((x+y)e_0+2 e_1\big),\quad&
    &q_3:=R\big((x-y)e_0\big)=R e_0
\end{aligned}
\end{equation*}
satisfy $\Harm(q_0,q_1,q_2,q_3)$. From (\ref{eq:beta.e}),
condition~(\ref{cond:ii}),
 and (\ref{eq:H.e}) follows
\begin{equation}\label{eq:mitte-mu}
\begin{aligned}
    &q_0^\mu=R (x^\beta e_0'+e_1'),&
    &q_1^\mu=R(y^\beta e_0'+e_1'),    \\
    &q_2^\mu=R\left(\left(\frac{x+y}{2}\right)^\beta e_0'+ e_1'\right),\quad&
    &q_3^\mu=R e_0'.
\end{aligned}
\end{equation}
We infer from (\ref{eq:dist-pr}) that $q_0^\mu \dis q_1^\mu$, and so $(x^\beta
e_0'+e_1',y^\beta e_0'+e_1')$ is a basis of $M'$. Now (\ref{eq:GL-R*}) yields
that $x^\beta-y^\beta$ is a unit in $R'$, whence
$q_3^\mu=R'\big((x^\beta-y^\beta)e_0'\big)$. By defining
\begin{equation}\label{eq:mitte'}
    q_2':=R'\big((x^\beta+y^\beta)e_0'+2 e_1'\big)
\end{equation}
we obtain $\Harm(q_0^\mu,q_1^\mu,q_2',q_3^\mu)$. The uniqueness of the third
harmonic point (see Section~\ref{se:basic}) shows $q_2'=q_2^\mu$. Comparing
(\ref{eq:mitte-mu}) with (\ref{eq:mitte'}) and taking into account
Lemma~\ref{lem:1} gives
\begin{equation}\label{eq:1/2}
    \left( \frac{x+y}{2} \right)^\beta = \frac{x^\beta+y^\beta}{2}
    \quad\mbox{for all}\quad x,y\in R
    \quad\mbox{with}\quad x-y\in R^*.
\end{equation}
Also $(R e_1)^\mu = R' e_1'$ implies $0^\beta=0$.
\par
Due to the last observation, condition~(\ref{cond:i}), Lemma~\ref{lem:1}, and
(\ref{eq:1/2}), we can apply the first part of \cite[Lemma~1]{lima+l-77b}. This
establishes that $\beta$ is additive. Moreover, (\ref{eq:H.e}) implies $1^\beta
= 1$.
\end{proof}

\begin{lem}\label{lem:3}
Let $(e_0,e_1)$, $(e_0',e_1')$ and $\beta$ be given as in
Lemma~{\rm\ref{lem:2}}. Let $t\in R$ be fixed. Then
\begin{equation}\label{eq:basis.f}
    (f_0,f_1):=(t e_0+ e_1,-e_0) \quad\mbox{and}\quad
    (f_0',f_1'):=(t^\beta e_0'+ e_1',-e_0')
\end{equation}
are bases of $M$ and $M'$, respectively, and there holds
\begin{equation}\label{eq:beta.f}
    \big(R(x f_0+ f_1)\big)^\mu = R'(x^\beta f_0'+ f_1')
    \quad\mbox{for all}\quad x\in R.
\end{equation}
\end{lem}

\begin{proof}
(a) The matrix of $(f_0,f_1)$ w.~r.~t.\ $(e_0,e_1)$ is $E(t)\in\E_2(R)$. So
$(f_0,f_1)$ is a basis of $M$. Likewise, $E(t^\beta)\in\E_2(R')$ shows that
$(f_0',f_1')$ is a basis of $M'$. We deduce $(Rf_0)^\mu = R' f_0'$ from
(\ref{eq:beta.e}), whereas (\ref{eq:H.e}) yields $(Rf_1)^\mu = R' f_1'$. Now
the additivity of $\beta$ together with $1^\beta=1$ gives
\begin{equation}\label{}
    \big(R(f_0\pm f_1)\big)^\mu = \Big(R\big( (t\mp 1)e_0 + e_1 \big)\Big)^\mu
        = R'\big( (t^\beta \mp 1)e_0' + e_1' \big)
        = R' (f_0'\pm f_1') .
\end{equation}
Consequently, as in Lemma~\ref{lem:2}, there is a unique mapping $\gamma:R\to
R'$ such that
\begin{equation}\label{eq:gamma.f}
    \big(R(x f_0+ f_1)\big)^\mu = R'(x^\gamma f_0'+ f_1')
    \quad\mbox{for all}\quad x\in R.
\end{equation}
Also, as before, $\gamma$ turns out to be additive with $1^\gamma=1$.
\par
(b) Consider a fixed $x\in R$ such that $1+x$ and $1-x$ are units. We define
\begin{equation*}\label{eq:g0123}
\begin{split}
    &g_0:=    (t+1)e_0+e_1 = f_0-f_1,\\
    &g_1:=    (t-1)e_0+e_1 = f_0+f_1,\\
    &g_2:=    2\big((t+x)e_0+e_1\big) = 2(f_0-x f_1),\\
    &g_3:=    2\big((1+xt)e_0+x e_1\big)   = 2(x f_0-f_1) .
\end{split}
\end{equation*}
The matrix of $(g_0,g_1)$ w.~r.~t.\ $(f_0,f_1)$ is in $\GL_2(R)$ due to $2\in
R^*$ and (\ref{eq:GL-R*}), whence $(g_0,g_1)$ is a basis. The equations
$(1+x)g_0+(1-x)g_1=g_2$ and $(1+x)g_0-(1-x)g_1=g_3$ yield that the points
$p_i:=Rg_i$, $i\in\{0,1,2,3\}$, satisfy $\Harm(p_0,p_1,p_2,p_3)$. We define
\begin{equation*}\label{eq:g0123'}
\begin{split}
    &g_0':=     \big((t^\beta+1)e_0'+e_1'\big) , \\
    &g_1':=     \big((t^\beta-1)e_0'+e_1'\big) ,\\
    &g_2':=    2\big((t^\beta+x^\beta)e_0'+e_1'\big) ,\\
    &g_3':=    2\big((1+x^\beta t^\beta)e_0'+x^\beta e_1'\big) ,
\end{split}
\end{equation*}
whence Lemma~\ref{lem:2} gives
\begin{equation}\label{eq:bilder}
     p_0^\mu= R'g_0' ,\quad  p_1^\mu= R'g_1' , \quad   p_2^\mu= R'g_2' .
\end{equation}
Now $\Harm(p_0^\mu,p_1^\mu,p_2^\mu,p_3^\mu)$ implies $p_0^\mu \dis p_2^\mu \dis
p_1^\mu$ so that
\begin{equation*}\label{eq:1+-x}
    \begin{pmatrix}
    t^\beta\ \pm 1 & 1 \\t^\beta +x^\beta  & 1
    \end{pmatrix}
    \in \GL_2(R')
\end{equation*}
which in turn, by (\ref{eq:GL-R*}), gives that $1+x^\beta$ and $1-x^\beta$ are
units in $R'$. We therefore are in a position to proceed as above in order to
establish $\Harm(R'g_0',R'g_1',R'g_2',R'g_3')$. By (\ref{eq:bilder}) and the
uniqueness of the fourth harmonic point, we obtain
\begin{equation*}\label{eq:p3.mu}
    p_3^\mu = R'g_3' = R'\big((1+x^\beta t^\beta)e_0'+x^\beta e_1'\big)
    = R'(x^\beta f_0' -  f_1').
\end{equation*}
On the other hand, writing $p_3=R\big((-x)f_0+f_1\big)$ allows us to apply
(\ref{eq:gamma.f}) which gives $p_3^\mu = R'\big((-x)^\gamma f_0' + f_1'\big)$.
The additivity of $\gamma$ yields
\begin{equation}\label{eq:beta=gamma.prov}
    x^\beta = x^\gamma \quad\mbox{for all}\quad  x\in R
    \quad\mbox{with}\quad 1+x
    \quad\mbox{and}\quad 1-x
    \quad\mbox{units.}
\end{equation}
\par
(c) If $x$ is any element of $R$ then, by condition~(\ref{cond:i}), there
exists $y\in R$ with $1+y$, $1-y$, $1+(x+y)$, and $1-(x+y)$ units. We infer
$y^\beta =y^\gamma$ and $(x+y)^\beta = (x+y)^\gamma$ from
(\ref{eq:beta=gamma.prov}) whence, by the additivity of $\beta$ and $\gamma$,
we obtain
\begin{equation*}\label{eq:beta=gamma}
    x^\beta = x^\gamma \quad\mbox{for all}\quad  x\in R.
\end{equation*}
This completes the proof of (\ref{eq:beta.f}).
\end{proof}

\begin{lem}\label{lem:4}
The mapping $\beta$ from Lemma~{\rm\ref{lem:2}} is a Jordan homomorphism.
\end{lem}

\begin{proof}
We make use of Lemma~\ref{lem:3} in the special case $t=0$, i.~e.,
$(f_0,f_1)=(e_1,-e_0)$ and $(f_0',f_1')=(e_1',-e_0')$. Given any $x\in R^*$ we
calculate the image of $R(xe_0+e_1) = R(-x^{-1}f_0 + f_1)$ according to
(\ref{eq:beta.e}) and (\ref{eq:beta.f}). This gives
\begin{equation*}\label{eq:invert}
    R'(x^\beta e_0'+e_1') =
    R'\big((-x^{-1} )^\beta f_0' + f_1'\big) =
    R'\big((-x^{-1})^\beta e_1' - e_0'\big).
\end{equation*}
Since $x^\beta e_0'+e_1'$ and $(-x^{-1})^\beta e_1' - e_0'$ are admissible
generators of the same point, there exists a unit $u' \in R'$ with $u'(x^\beta
e_0'+e_1') = (-x^{-1})^\beta e_1' - e_0'$. Now $u'x^\beta=-1$ implies that
$x^\beta$ is a unit in $R'$ and, by the additivity of $\beta$, we obtain
\begin{equation}\label{eq:inv}
    (x^\beta)^{-1} = (x^{-1})^\beta \quad\mbox{for all}\quad  x\in R^*.
\end{equation}
Due to $1^\beta=1$ and (\ref{eq:inv}) we are in a position to apply also the
second part of \cite[Lemma~1]{lima+l-77b} which establishes that $\beta$
satisfies
\begin{equation}\label{eq:xy+yx}
    (xy+yx)^\beta = x^\beta y^\beta +y^\beta x^\beta
    \quad\mbox{for all}\quad  x,y\in R.
\end{equation}
Recall that $2$ is a unit in $R$ by condition~(\ref{cond:ii}), and also a unit
in $R'$ by Lemma~\ref{lem:1}. Moreover, from Lemma~\ref{lem:2}, $\beta$ is
additive and satisfies $1^\beta=1$. It is well known that under these
circumstances (\ref{eq:xy+yx}) characterises $\beta$ as being a Jordan
homomorphism; see, e.~g., \cite[p.~47]{herst-69a} or \cite[p.~320]{hua-52a}.
\end{proof}

\begin{proof}[Proof of Theorem {\rm\ref{thm:staudt}}]
Choose any point of the connected component $C$, say $p_0=R a_{0}$, and any
$a_{1}\in M$ such that $(a_{0},a_{1})$ is a basis of $M$. Let $p_1:=R a_{1}$,
$p_2:= R(a_{0}+a_{1})$, and $p_3:= R(a_{0}-a_{1})$. Then
$\Harm(p_0,p_1,p_2,p_3)$ implies $\Harm(p_0^\mu,p_1^\mu,p_2^\mu,p_3^\mu)$ so
that there exists a basis $(a_{0}',a_{1}')$ of $M'$ satisfying
\begin{equation}\label{eq:H.a}
    (Ra_{0})^\mu= R' a_{0}',\quad
    (Ra_{1})^\mu=R' a_{1}',\quad
    \big(R(a_{0}\pm a_{0})\big)^\mu=R' (a_{0}'\pm a_{1}') .
\end{equation}
We apply Lemma~\ref{lem:2} to the bases $(a_0,a_1)$ and $(a_0',a_1')$, but
relabel the mapping $\beta$ from there as $\alpha$. So, by (\ref{eq:beta.e})
and Lemma~\ref{lem:4}, there exists a Jordan homomorphism $\alpha:R\to R'$ with
\begin{equation}\label{eq:alpha.a}
    \big(R(x a_0+ a_1)\big)^\mu = R'(x^\alpha a_0'+ a_1')
    \quad\mbox{for all}\quad x\in R .
\end{equation}
By (\ref{eq:component}), a point $p\in \bP(M)$ belongs to $C$ precisely when
there is at least one sequence $T\in\Seq(R)$ such that $p=R(x_0 a_{0} + x_1
a_{1})$ with $(x_0, x_1)=(1,0)\cdot E(T)$. It therefore remains to verify that
for all finite sequences $T\in\Seq(R)$ the coordinate rows $(x_0,x_1)$ and
$(x_0',x_1')$ from (\ref{eq:(x0,x1)(x0',x1')}) define points which correspond
under $\mu$. We proceed by induction on the length of $T$ which will be denoted
by $n$.
\par
For $n=0$ the sequence $T$ is empty and $E()$ is the identity matrix. Now
(\ref{eq:(x0,x1)(x0',x1')}) reads $(x_0, x_1) = (1,0)$, $(x_0', x_1') = (1,0)$,
and indeed $(R a_{0})^\mu = R'a_{0}'$ according to (\ref{eq:H.a}).
\par
For $n=1$ we have $T=(t_1)$ with $t_1\in R$. The assertion follows from
(\ref{eq:alpha.a}), since (\ref{eq:(x0,x1)(x0',x1')}) now takes the form $(x_0,
x_1) = (t_1,1)$, $(x_0', x_1') = (t_1^\alpha,1)$.
\par
Let $n\geq 2$ and suppose $T=(t_1,t_2,\ldots,t_n)\in\Seq(R)$. There is a unique
basis of $M$, say $(e_0,e_1)$, with $E(t_3,\ldots,t_n)$ being its matrix
w.~r.~t.\ $(a_0,a_1)$. We proceed analogously in $M'$ and obtain a basis
$(e_0',e_1')$ with $E(t_3^\alpha,\ldots,t_n^\alpha)$ being its matrix w.~r.~t.\
$(a_0',a_1')$. The following table displays for all $x\in R$ the coordinates of
certain elements of $M$ and $M'$:
\begin{equation}\label{eq:tabelle.e}
    \begin{array}{|c|l|c|l|}
    \hline
    \multicolumn{2}{|c|}{\mbox{Coordinates w.~r.~t.~} (a_0,a_1)}&
    \multicolumn{2}{|c|}{\mbox{Coordinates w.~r.~t.~} (a_0',a_1')}\\
      \hline
        e_0 & (1,0)\cdot E(t_3,\ldots,t_n) &
        e_0' & (1,0)\cdot E(t_3^\alpha,\ldots,t_n^\alpha)\\
        x e_0+e_1 & (1,0)\cdot E(x,t_3,\ldots,t_n) & x^\alpha e_0'+e_1' &
        (1,0)\cdot E(x^\alpha,t_3^\alpha,\ldots,t_n^\alpha)\\
    \hline
    \end{array}
\end{equation}
Those elements of $M$ and $M'$ which appear in the same row of table
(\ref{eq:tabelle.e}) generate corresponding points under $\mu$ due to the
induction hypothesis. In particular, as $x$ ranges in $\{0,1,-1\}$, we get
\begin{equation*}\label{eq:H.e-neu}
    (Re_0)^\mu= R'e_0',\quad
    (Re_1)^\mu= R'e_1', \quad
    \big(R(\pm e_0+e_1)\big)^\mu= R'(\pm e_0'+e_1').
\end{equation*}
Hence the bases $(e_0,e_1)$ and $(e_0',e_1')$ satisfy (\ref{eq:H.e}) so that
Lemma~\ref{lem:2} can be applied to them (without any notational changes). We
claim that $\alpha$, as defined via (\ref{eq:alpha.a}), coincides with the
Jordan homomorphism $\beta$ appearing in Lemma~\ref{lem:2}: Indeed, $\alpha$
satisfies the defining equation (\ref{eq:beta.e}) according to the second row
of table (\ref{eq:tabelle.e}) in conjunction with the induction hypothesis. We
now introduce bases $(f_0,f_1)$ of $M$ and $(f_0',f_1')$ of $M'$ as in
Lemma~\ref{lem:3}, but replace the arbitrary $t\in R$ from there by the given
$t_2\in R$. This gives a second table of coordinates:
\begin{equation}\label{eq:tabelle.f}
    \begin{array}{|c|l|c|l|}
    \hline
    \multicolumn{2}{|c|}{\mbox{Coordinates w.~r.~t.~} (a_0,a_1)}&
    \multicolumn{2}{|c|}{\mbox{Coordinates w.~r.~t.~} (a_0',a_1')}\\
      \hline
      f_0 & (1,0)\cdot E(t_2,\ldots,t_n) &f_0' & (1,0)\cdot E(t_2^\alpha,\ldots,t_n^\alpha)\\
    t_1 f_0+f_1 & (1,0)\cdot E(t_1,t_2,\ldots,t_n) &    t_1^\alpha f_0'+f_1' & (1,0)\cdot E(t_1^\alpha,t_2^\alpha,\ldots,t_n^\alpha)\\
    \hline
    \end{array}
\end{equation}
Since $\alpha=\beta$, we can read off from (\ref{eq:beta.f}) that
$\big(R(t_1f_0+f_1)\big)^\mu = R'(t_1^\alpha f_0'+f_1')$. Hence the coordinates
from the last row of table (\ref{eq:tabelle.f}) describe points which
correspond under $\mu$.
\end{proof}

\paragraph{Acknowledgements.} The author expresses his warmest thanks to Mark Pankov (Olsztyn) for the translation
of several articles from Russian to English.



\noindent
Hans Havlicek\\
Institut f\"{u}r Diskrete Mathematik und Geometrie\\
Technische Universit\"{a}t\\
Wiedner Hauptstra{\ss}e 8--10/104\\
A-1040 Wien\\
Austria\\
\texttt{havlicek@geometrie.tuwien.ac.at}
\end{document}